\documentclass{amsart}
  \RequirePackage{amsmath}
\RequirePackage{amssymb} \RequirePackage{amsthm}
\RequirePackage{epsfig} 
                \def\8{\theta}

\newcommand{\hol}{{\mathcal Hol}}

\newcommand{\h}{\mathcal{H}}

\def\D{{\mathbb D}}
\def\T{{\mathbb T}}
\def\C{{\mathbb C}}

\def\({\left(}       \def\){\right)}

\def\ol{\overline}

\newcommand{\ig}{\stackrel{\text{def}}{=}}

\newcommand{\equstart}{\begin{equation}\begin{aligned}}
\newcommand{\equend}{\end{aligned}\end{equation}}
\newcommand{\equstartu}{\begin{equation*}\begin{aligned}}
\newcommand{\equendu}{\end{aligned}\end{equation*}}
\newtheorem{theorem}{Theorem}
\newtheorem{lemma}{Lemma}

\newtheorem{proposition}{Proposition}
\theoremstyle{definition}

\theoremstyle{remark}

\numberwithin{equation}{section}
\theoremstyle{theorem}
\newtheorem{other}{\bf Theorem}              
\newtheorem{otherl}{\bf Lemma}        

\newenvironment{pf}{\noindent{\emph{Proof.}}}{$\Box$ }
\newenvironment{Pf}{\noindent{\emph{Proof of}}}{$\Box$ }
\DeclareMathOperator{\Imag}{Im}

\begin{document}
\title[A generalized Hilbert matrix acting on Hardy spaces]
{A generalized Hilbert matrix acting on Hardy spaces}

\author[Ch.~Chatzifountas]{Christos Chatzifountas}
\address{Departamento de An\'alisis Matem\'atico,
Universidad de M\'alaga, Campus de Teatinos, 29071 M\'alaga, Spain}
 \email{christos.ch@uma.es}
\author[D.~Girela]{Daniel Girela}
 \address{Departamento de An\'alisis Matem\'atico,
Universidad de M\'alaga, Campus de Teatinos, 29071 M\'alaga, Spain}
 \email{girela@uma.es}
\author[J.~A.~Pel\'aez]{Jos\'e \'Angel Pel\'aez}
 \address{Departamento de An\'alisis Matem\'atico,
Universidad de M\'alaga, Campus de Teatinos, 29071 M\'alaga, Spain}
 \email{japelaez@uma.es}

\date{July 4, 2013} 

\begin{abstract}\par If $\mu $ is a positive Borel measure on the interval $[0, 1)$,
 the Hankel matrix $\mathcal H_\mu =(\mu _{n,k})_{n,k\ge
0}$  with entries $\mu _{n,k}=\int_{[0,1)}t^{n+k}\,d\mu(t)$  induces
formally the operator
$$\mathcal{H}_\mu (f)(z)=
\sum_{n=0}^{\infty}\left(\sum_{k=0}^{\infty}
\mu_{n,k}{a_k}\right)z^n$$ on the space of all analytic functions
$f(z)=\sum_{k=0}^\infty a_kz^k$, in the unit disc $\D $. In this
paper we describe those measures $\mu$ for which $\h_\mu $ is a
bounded (compact) operator from $H^p$ into $H^q$,  $0<p,q<\infty $.
We  also characterize the measures $\mu $ for which $\mathcal H_\mu
$ lies in the Schatten class $S_p(H^2)$, $1<p<\infty$.
\end{abstract}

\thanks{This
research is supported by a grant from la Direcci\'{o}n General de
Investigaci\'{o}n, Spain (MTM2011-25502) and by a grant from la
Junta de Andaluc\'{\i}a (P09-FQM-4468 and FQM-210). The third author
is supported also by the \lq\lq Ram\'on y Cajal program\rq\rq ,
Spain.}

\maketitle

\bigskip
\section{Introduction and main results}\label{intro}
Let $\D=\{z\in\C: |z|<1\}$ denote the open unit disc in the complex
plane $\C$ and let $\hol (\D)$ be the space of all analytic
functions in $\D$. We also let  $H^p$ ($0<p\le \infty $) be the
classical Hardy spaces (see \cite{D}).
\par
If\, $\mu $ is a finite positive Borel measure on $[0, 1)$ and $n\,
= 0, 1, 2, \dots $, we let $\mu_n$ denote the moment of order $n$ of
$\mu $, that is,
$$\mu _n=\int _{[0,1)}t^n\,d\mu (t),$$
and we define $\mathcal H_\mu $ to be the Hankel matrix $(\mu
_{n,k})_{n,k\ge 0}$ with entries $\mu _{n,k}=\mu_{n+k}$. The matrix
$\mathcal H_\mu $ can be viewed as an operator on spaces of analytic
functions by its action on the Taylor coefficients: \,$ a_n\mapsto
\sum_{k=0}^{\infty} \mu_{n,k}{a_k}, \quad n=0,1,2, \cdots . $\, To
be precise, if\,
 $f(z)=\sum_{k=0}^\infty a_kz^k\in \hol (\D )$
we define
\begin{equation}\label{H}
\mathcal{H}_\mu (f)(z)= \sum_{n=0}^{\infty}\left(\sum_{k=0}^{\infty}
\mu_{n,k}{a_k}\right)z^n,
\end{equation}
whenever the right hand side makes sense and defines an analytic
function in $\D $.
\par\medskip If $\mu $ is the Lebesgue measure on $[0,1)$ the matrix
$\mathcal H_\mu $ reduces to the classical Hilbert matrix \,
$\mathcal H= \left ({(n+k+1)^{-1}}\right )_{n,k\ge 0}$, which
induces the classical Hilbert operator $\h$, a prototype of a Hankel
operator which has attracted a considerable amount of attention
during the last years. Indeed, the study of  the boundedness, the
operator norm and the spectrum of $\h$  on Hardy and weighted
Bergman spaces \cite{AlMonSa,DiS,DJV,GaGiPeSis,PelRathg} links $\h$
up to weighted composition operators, the Szeg\"{o} projection,
Legendre functions and the theory of Muckenhoupt weights.
\par Hardy's
inequality \cite[page~48]{D} guarantees that  $\mathcal{H}(f)$ is a
well defined analytic function in $\D$ for every $f\in H^1$.
However, the resulting Hilbert operator $\mathcal{H}$ is bounded
from $H^p$ to $H^p$ if and only if $1<p<\infty$ \cite{DiS}.
In a recent paper \cite{LNP} Lanucha, Nowak, and Pavlovic have
considered the question of finding subspaces of $H^1$ which are
mapped by $\mathcal{H} $ into $H^1$.
\par Galanopoulos and Pel\'{a}ez \cite{Ga-Pe2010} have described
the measures $\mu $ so that the generalized Hilbert operator $\mathcal H_\mu
$ becomes well defined and bounded on $H^1$. Carleson measures
 play
a basic role in the work.
\par If $I\subset \partial\D$ is an
interval, $\vert I\vert $ will denote the length of $I$. The
\emph{Carleson square} $S(I)$ is defined as
$S(I)=\{re^{it}:\,e^{it}\in I,\quad 1-\frac{|I|}{2\pi }\le r <1\}$.
Also, for $a\in \D$, the Carleson box $S(a)$ is defined by
\begin{displaymath}
S(a)=\Big \{ z\in \D : 1-|z|\leq 1-|a|,\, \Big |\frac{\arg
(a\bar{z})}{2\pi}\Big |\leq \frac{1-|a|}{2} \Big \}.
\end{displaymath}
\par If $\, s>0$ and $\mu$ is a positive Borel  measure on  $\D$,
we shall say that $\mu $
 is an $s$-Carleson measure
  if there exists a positive constant $C$ such that
\[
\mu\left(S(I)\right )\le C{|I|^s}, \quad\hbox{for any interval
$I\subset\partial\D $},
\]
or, equivalently,
 if there exists $C>0$ such that
\[
\mu\left(S(a)\right )\le C{(1-\vert a\vert )^s}, \quad\hbox{for all
$a\in\D$}.
\]
If $\mu $ satisfies $\displaystyle{\lim_{\vert I\vert\to 0}\frac{\mu
\left (S(I)\right )}{\vert I\vert ^s}=0}$ or, equivalently,
$\displaystyle{\lim_{\vert a\vert \to 1}\frac{\mu \left (S(a)\right
)}{(1-\vert a\vert^2)^s}=0}$, then we say that $\mu $ is a\, {\it
vanishing $s$-Carleson measure}. \par  An $1$-Carleson measure,
respectively, a vanishing $1$-Carleson measure, will be simply
called a Carleson measure, respectively, a vanishing Carleson
measure.
\par As an important ingredient in his work on interpolation by
bounded analytic functions, Carleson \cite{Ca2} (see also Theorem
9.3 of \cite{D}) proved that if $0<p<\infty $ and $\mu$ is a
positive Borel measure in $\D $ then $H\sp p\subset L\sp p(d\mu )$
if and only if \,$\mu $\, is a Carleson measure. This result was
extended by Duren \cite{Du:Ca} (see also \cite[Theorem~9.\@4]{D})
who proved that for $\,0<p\le q<\infty $,
 $H^p\subset L^q(d\mu)$
if and only if $\mu$ is a $q/p$-Carleson measure. \par If $X$ is a
subspace of $\hol (\D )$, $0<q<\infty $, and $\mu $ is a positive
Borel measure in $\D $, $\mu $ is said to be a \lq\lq ${q}${\it
{-Carleson measure for the space}} ${X}$\rq\rq \, or an \lq\lq
${{(X, q)}}${\it-Carleson measure}\rq\rq \, if $X\subset L^q(d\mu
)$. The $q$-Carleson measures for the spaces $H^p$, $0<p,q<\infty $
are completely characterized. The mentioned results of Carleson and
Duren can be stated saying that if $\,0<p\le q<\infty $\, then a
positive Borel measure $\mu $ in $\D$ is a $q$-Carleson measure for
$H^p$ if and only if $\mu$ is a $q/p$-Carleson measure. Luecking
\cite{Lu90} and Videnskii \cite{Vid} solved the remaining case
$0<q<p$. We mention \cite{Bl-Ja} for a complete information on
Carleson measures for Hardy spaces.

\par\medskip Galanopoulos and Pel\'{a}ez proved in \cite{Ga-Pe2010}
that if\, $\mu $\, is a Carleson measure then the operator $\mathcal
H_\mu $\,  is well defined in $H^1$, obtaining en route the
following integral representation
\begin{equation}\label{HmuImuH1}\mathcal H_\mu
(f)(z)\,=\,\int_{[0,1)}\frac{f(t)}{1-tz}\,d\mu (t),\quad z\in
\mathbb D,\quad\text{for all $f\in H^1$}.\end{equation} For
simplicity, we shall write throughout the paper
\begin{equation}\label{Imu}I_\mu
(f)(z)=\int_{[0,1)}\frac{f(t)}{1-tz}\,d\mu (t),\end{equation}
whenever the right hand side makes sense and defines an analytic
function in $\D $. It was also proved in \cite{Ga-Pe2010} that if
$I_\mu (f)$ defines an analytic function in $\D $ for all $f\in
H^1$, then $\mu $ has to be a Carleson measure. This condition  does
not ensures the boundedness of $\h_\mu$ on $H^1$, as the classical
Hilbert operator $\mathcal H$ shows.

\par\medskip
Let $\mu$ be a positive Borel measure in $\D$, $0\le \alpha <\infty
$, and $0<s<\infty $. Following \cite{Zhao}, we say that $\mu$ is an
 $\alpha$-logarithmic $s$-Carleson measure, respectively, a vanishing $\alpha$-logarithmic $s$-Carleson measure, if
\begin{displaymath}
\sup_{a\in\D}\frac{\mu\left(S(a)\right)\left (\log
\frac{2}{1-|a|^2}\right )^{\alpha}}{(1-|a|^2)^s}<\infty,
\quad\text{respectively,}\,\,\,\lim_{|a|\to
1^{-}}\frac{\mu\left(S(a)\right)\left (\log \frac{2}{1-|a|^2}\right
)^{\alpha}}{(1-|a|^2)^s}=0. \end{displaymath} \par
Theorem\,\@1.\,\@2 of \cite{Ga-Pe2010} asserts that if $\mu $ is a
Carleson measure on $[0,1)$, then $\mathcal H_\mu$ is a bounded
(respectively, compact) operator from $H^1$ into $H^1$ if and only
if $\mu $ is a $1$-logarithmic $1$-Carleson measure (respectively, a
vanishing $1$-logarithmic $1$-Carleson measure). \par It is also
known that $\mathcal H_\mu $ is bounded from $H^2$ into itself if
and only if $\mu $ is a Carleson measure (see \cite[p. 42, Theorem
$7.2$]{Pell}).

\par\medskip Our
main aim in this paper is to study the generalized Hilbert matrix
$\mathcal H_\mu $ acting on $H^p$ spaces $(0<p<\infty $). Namely,
for any given $p, q$ with $0<p, q<\infty $, we wish to characterize
those for which $\h_\mu $ is a bounded (compact) operator from $H^p$
into $H^q$ and to describe those measures $\mu$ such that $\h_\mu$
belongs to the Schatten class $\mathcal S_p(H^2)$. A key tool will
be a description of those positive Borel measures $\mu $ on $[0,1)$
for which $\h_\mu $ is well defined in $H^p$ and
$\h_\mu(f)=I_\mu(f)$.
 Let us
start with the case $p\le 1$.

\begin{theorem}\label{Def:p<1} Suppose that $0<p\le 1$ and let $\mu
$ be a positive Borel measure on $[0,1)$. Then the following two
conditions are equivalent:
\begin{itemize}
\item[(i)] $\mu $ is an
$\frac{1}{p}$-Carleson measure.
\item[(ii)] $I_\mu (f)$ is a well
defined analytic function in $\D $ for any $f\in H^p$.
\end{itemize}

\par Furthermore, if (i) and (ii) hold and $f\in H^p$, then $\mathcal H_\mu (f)$  is
also a well defined analytic function in $\D $, and $\mathcal H_\mu
(f)=I_\mu (f)$, for all $f\in H^p.$
\end{theorem}
\par\medskip
We remark that for $p=1$, this reduces to
\cite[Proposition\,\@1.\,\@1]{Ga-Pe2010}. \par\medskip For $0<q<1$,
we let $B_q$ denote the space consisting of those $f\in \hol (\D )$
for which
$$\int_0^1\,(1-r)^{\frac{1}{q}-2}M_1(r,f)\,dr<\infty .$$ The Banach space
$B_q$ is the \lq\lq containing Banach space\rq\rq \, of $H^q$, that
is, $H^q$ is a dense subspace of $B_q$, and the two spaces have the
same continuous linear functionals \cite{DRS}. Next we shall show
that if $\mu $ is an $1/p$-Carleson measure then $\mathcal H_\mu $
actually applies $H^p$ into $B_q$ for all $q<1$. We shall also give
a characterization of those $\mu $ for which $\mathcal H_\mu $ map
$H^p$ into $H^q$ ($q\ge 1$). Before stating these results precisely,
let us mention that all over the paper we shall use the notation
that for any given $\alpha
>1$, $\alpha^\prime $ will denote the conjugate exponent of $\alpha
$, that is, $\frac{1}{\alpha }+\frac{1}{\alpha ^\prime }=1$, or
$\alpha^\prime =\frac{\alpha }{\alpha -1}$.

\begin{theorem}\label{bound:p<1}
 Suppose that $0<p\le 1$ and let $\mu $
be a positive Borel measure on $[0,1)$ which is an
$\frac{1}{p}$-Carleson measure.
\begin{itemize}\item[(i)] If\, $0<q<1$, then $\h_\mu $ is a bounded
operator from $H^p$ into $B_q$, the containing Banach space of
$H^q$.
\item[(ii)] $\h_\mu $ is a bounded
operator from $H^p$ into $H^1$ if and only if $\mu $ is an
$1$-logarithmic $\frac{1}{p}$-Carleson measure.
\item[(iii)] If\, $q>1$ then
$\h_\mu $ is a bounded operator from $H^p$ into $H^q$ if and only if
$\mu $ is an $\frac{1}{p}+\frac{1}{q^\prime}$-Carleson measure.
\end{itemize}
\end{theorem}
\par\medskip Let us state next our results for $p>1$.
\begin{theorem}\label{Def:p>1}
Suppose that $1<p<\infty $ and let $\mu $ be a positive Borel
measure on $[0,1)$. Then:
\par (i) $I_\mu (f)$ is a well defined analytic function in $\D $
for any $f\in H^p$ if and only if $\mu $ is an $1$-Carleson measure
for $H^p$, or, equivalently, if and only if
\begin{equation}\label{int-car-1p}\int_0^1\,\left (\int_0^{1-s}\,\frac{d\mu (t)}{1-t}\right
)^{p^\prime }\,ds\,<\,\infty .\end{equation}
\par (ii) If $\mu $ satisfies
(\ref{int-car-1p}) then
 $\mathcal H_\mu (f)$ is also a well defined analytic function in $\D
 $,
whenever $f\in H^p$, and
 $$\mathcal H_\mu (f)=I_\mu (f),\quad\text{for every
$f\in H^p$.}$$
\end{theorem}
\par\medskip
\begin{theorem}\label{bound:p>1}
Suppose that $1<p<\infty $ and let $\mu $ be a positive Borel
measure on $[0,1)$ which satisfies .
\begin{itemize}\item[(i)] If $0<p\le q<\infty $, then $\mathcal H_\mu $
is a bounded operator from $H^p$ to $H^q$ if and only if $\mu $ is
an $\frac{1}{p}+\frac{1}{q^\prime }$-Carleson measure.
\item[(ii)] If $1<q<p$, then $\mathcal H_\mu $
is a bounded operator from $H^p$ to $H^q$ if and only if the
function defined by \,\,$s\mapsto \int_0^{1-s}\,\frac{d\mu
(t)}{1-t}$\,\, $(s\in [0,1))$\, belongs to $L^{\left
(\frac{pq^\prime }{p+q^\prime }\right )^\prime }([0,1))$.
\item[(iii)] $\mathcal H_\mu $
is a bounded operator from $H^p$ to $H^1$ if and only if the
function defined by \,\,$s\mapsto
\int_0^{1-s}\,\frac{\log\frac{1}{1-t}d\mu (t)}{1-t}$\,\, $(s\in
[0,1))$\, belongs to $L^{p^\prime }([0,1))$.
\item[(iv)] If $0<q<1$, then $\mathcal H_\mu $
is a bounded operator from $H^p$ into $B_q$.
\end{itemize}
\end{theorem}
\par\medskip
Let us remark that both if either $0<p\le 1$ and $\mu $ is an
$1/p$-Carleson measure, or if $1<p<\infty $ and $\mu $ satisfies
(\ref{int-car-1p}), we have that $\mu $ is an $1$-Carleson measure
for $H^p$. By the closed graph theorem it follows
that, for any $q>0$,
$$\mathcal H_\mu (H^p)\subset H^q\,\,\,\Leftrightarrow\,\,\,
\mathcal H_\mu \,\,\,\text{is a bounded operator form $H^p$ into
$H^q$.}$$
\par\medskip Substitutes of Theorem\,\@\ref{bound:p<1} and
Theorem\,\@\ref{bound:p>1} regarding compactness will be stated and
proved in Section\,\@\ref{sect-compact}.

\par\medskip Finally, we address the question of describing those measures $\mu$ such that $\h_\mu$ belongs to the Schatten class
$\mathcal S_p(H^2)$, ($1<p<\infty $). Given a separable Hilbert
space $X$ and
 $0<p<\infty$, let $\mathcal{S}_ p(X)$  denote the
Schatten $p$-class of operators on $X$. The class $\mathcal{S}_
p(X)$ consists of those compact operators $T$ on $X$ whose sequence
of singular numbers $\{ \lambda_ n\} $ belongs to $\ell^p$, the
space of $p$-summable sequences. It is well known that, if $\lambda_
n$ are the singular numbers of an operator $T$, then
\begin{displaymath}
\lambda_ n=\lambda_ n(T)=\inf \{\|T-K\|: \textrm{ rank}\,K\leq n\}.
\end{displaymath}
Thus finite rank operators belong to every $\mathcal{S}_ p(X)$, and
the membership of an operator in $\mathcal{S}_ p(X)$ measures in
some sense the size of the operator. In the case when $1\leq
p<\infty$, $\mathcal{S}_ p(X)$ is a Banach space with the norm
\begin{displaymath}
\|T\|_ p=\left(\sum_ n |\lambda_ n|^p\right )^{1/p},
\end{displaymath}
while for $0<p<1$ we have the following inequality $\|
T+S\|_p^p\le\|T\|_ p^p+\|S\|_ p^p.$ We refer to \cite{Zhu} for more
information about  $\mathcal{S}_ p(X)$.

\par\medskip Galanopoulos
and Pel\'{a}ez \cite[Theorem\,\@1.\,\@6]{Ga-Pe2010} found a
characterization of those $\mu $ for which $\h_\mu $ is a
Hilbert-Schmidt operator on $H^2$ improving a result of \cite{Pow}.
In \cite[p. $239$, Corollary\,\@2.\,\@2]{Pell} it is proved that,
for $1<p<\infty $, $\mathcal H_\mu \in \mathcal S_p(H^2)$ if and
only if $h_\mu(z)=\sum_{n=1}^\infty \mu_{n+1}z^{n}$ belongs to the
Besov space $B^p$ (see \cite[Chapter\,\@5]{Zhu})  of those analytic
functions $g$ in $\D$ such that
$$||g||^p_{B^p}=|g(0)|^p+\int_\D |g'(z)|^p(1-|z|^2)^{p-2}\,dA(z)<\infty.$$
  We simplify  this result
 describing the membership of
$\mathcal H_\mu $  in the Schatten class $\mathcal S_p(H^2)$
 in terms of the moments $\mu_n$.
\begin{theorem}\label{th:Schatten}
Assume that $1<p<\infty$ and let $\mu$ be a positive Borel measure
on $[0,1)$. Then, $\mathcal H_\mu\in \mathcal S_p(H^2)$ if and only
if $\sum_{n=0}^\infty (n+1)^{p-1}\mu_n^p<\infty.$
\end{theorem}
\par\bigskip
Throughout the paper, the letter $C=C(\cdot)$ will denote a constant
whose value depends on the parameters indicated in the parenthesis
(which often will be omitted), and may change from one occurrence to
another. We will use the notation $a\lesssim b$ if there exists
$C=C(\cdot)>0$ such that $a\le Cb$, and $a\gtrsim b$ is understood
in an analogous manner. In particular, if $a\lesssim b$ and
$a\gtrsim b$, then we will write $a\asymp b$.

\section{Preliminary results} In this section we shall collect a
number of results which will be needed in our work. We start
obtaining a characterization of $s$-Carleson measures in terms of
the moments.
\begin{proposition}\label{moments-cond} Let $\mu $ be a positive Borel measure on
$[0,1)$ and $s>0$. Then $\mu $ is an $s$-Carleson measure if and
only if the sequence of moments $\{ \mu _n\} _{n=0}^\infty $
satisfies
\begin{equation}\label{cond-moments} \sup _{n\ge
0}\,(1+n)^s\,\mu_n<\infty .\end{equation}
\end{proposition}
The proof is  simple and will be omitted.
\par\medskip The following result, which may be of independent interest, asserts that for any function $f\in
H^p$ ($0<p<\infty $) we can find another one $F$ with the same
$H^p$-norm and which is non-negative and bigger than $\vert f\vert $
in the radius $(0,1)$.
\begin{proposition}\label{may-Hp} Suppose that $0<p<\infty $ and
$f\in H^p$, $f\not\equiv 0$. Then there exists a function $F\in H^p$
with $\Vert F\Vert _{H^p}= \Vert f\Vert _{H^p}$ and satisfying the
following properties:
\par (i) $F(r)>0$, for all $r\in (0, 1)$.
\par (ii) $\vert f(r)\vert \,\le \,F(r)$, for all $r\in (0, 1)$.
\par (iii) $F$ has no zeros in $\mathbb D$.
\end{proposition}
\begin{proof}
Let us consider first the case $p=2$. So, take
$f(z)=\sum_{n=0}^\infty a_nz^n\,\in H^2$, $f\not \equiv 0$. Set
$G(z)=\sum _{n=0}^\infty \vert a_n\vert z^n$ ($z\in \D $). Then
$G\in H^2$ and $\Vert G\Vert_{H^2}=\Vert f\Vert_{H^2}$. Furthermore,
we have: \begin{equation}\label{G(r)}0\le \vert f(r)\vert \le
G(r)\,\,\,\text{and $G(r)>0$},\quad \text{for all $r\in (0,
1)$},\end{equation} and
\begin{equation}\label{G-real-coef}G(\overline z)=\overline
{G(z)},\quad z\in \D .\end{equation} By (\ref{G(r)}) and
(\ref{G-real-coef}) we  see that the sequence $\{ z_n\} $ of
the zeros of $G$ with $z_n\neq 0$ (which is a Blaschke sequence) can
be written in the form $\{ z_n\} =\{ \alpha _n\} \cup \{ \overline
{\alpha _n}\} \cup \{ \beta _n\} $ where $\Imag (\alpha_n)>0$ and
$-1<\beta _n<0$. Then the Blaschke product $B$ with the same zeros
that $G$ is
$$B(z)=z^m\prod \left (\frac{\alpha _n-z}{1-\overline {\alpha
_n}z}\frac{\overline {\alpha _n}-z}{1-\alpha _nz}\right )\prod
\frac{z-\beta_n}{1-\beta_nz},$$ where $m$ is the order of $0$ as
zero of $G$ (maybe $0$). Using the Riesz factorization theorem
\cite[Theorem\,\@2.\,\@5]{D}, we can factor $G$ in the form
$G=B\cdot F$ where $F$ is an $H^2$-function with no zeros and with
$\Vert f\Vert _{H^2}=\Vert G\Vert _{H^2}=\Vert F\Vert _{H^2}$.
Notice that $B(r)>0$, for all $r\in (0,1)$.  This together with
(\ref{G(r)}) gives  that $F(r)>0$, for all $r\in (0,1)$. Finally
since $\vert B(z)\vert \le 1$, for all $z$, we have that $G(r)\le
F(r)$ ($r\in (0,1)$) and then (\ref{G(r)}) implies $\vert f(r)\vert
\le F(r)$ ($r\in (0,1)$). This finishes the proof in the case $p=2$.
\medskip\par If $0<p<\infty $ and $f\in H^p$, $f\not \equiv 0$, write $f$ in
the form $f=B\cdot g$ where $B$ is a Blaschke product and $g$ is and
$H^p$-function without zeros and with $\Vert g\Vert _{H^p}=\Vert
f\Vert _{H^p}$. Now $g^{p/2}\in H^2$. By the previous case, we have
a function $G\in H^2$ without zeros, which take positive values in
the radius $(0,1)$, and satisfying $\Vert G\Vert _{H^2}=\Vert
g^{p/2}\Vert _{H^2}$ and $\vert g(r)\vert ^{p/2}\le G(r)$, for all
$r\in (0,1)$. It is clear that the function $F=G^{2/p}$ satisfies
that conclusion of Proposition~\ref{may-Hp}.
\end{proof}
\par\medskip
We shall also use the following description of $\alpha$-logarithmic
$s$-Carleson measures (see  \cite[Theorem\,\@2]{Zhao}).
\begin{otherl}\label{le:zh}
Suppose that $0\le \alpha<\infty$ and $0<s<\infty$ and $\mu$ is a
positive Borel measure in $\D$. Then $\mu$ is an
$\alpha$-logarithmic $s$-Carleson measure if and only if
\begin{equation}\label{ELC}
K_{\alpha,s}(\mu)\ig\sup_{a\in \D}\left(\log \frac{2}{1-|a|^2}\right
)^{\alpha}\,\int_{\D}\left (\frac{1-|a|^2}{|1-\bar{a}z|^2}\right
)^s\,d\mu(z)<\infty.
\end{equation}
When $\alpha =0$, the constant $K_{0,s}(\mu )$ will be simply
written as $K_s(\mu )$. We remark that, if $s\ge 1$, then $K_s(\mu
)$ is equivalent to the norm of the embedding \,\,$i:H^p\rightarrow
L^{ps}(d\mu )$ for any $p\in (0,\infty )$.
\end{otherl}
\par\bigskip
Next we recall the following useful characterization of $q$-Carleson
measures for $H^p$ in the case $0<q<p<\infty $ (see \cite{{Bl-Ja},
{Lu90}, {Vid}}). \par Take $\alpha $ with $0<\alpha <\frac{\pi
}{2}$. Given $s\in \mathbb R$, we let $\Gamma_\alpha (e^{is})$
denote the Stolz angle with vertex $e^{is}$ and semi-aperture
$\alpha $, that is, the interior of the convex hull of $\{ e^{is}\}
\cup \{ \vert z\vert <\sin \alpha \} $. If $\mu $ is a positive
Borel measure in $\D $, we define \lq\lq the $\alpha
$-balagaye\rq\rq \, $\tilde \mu_\alpha $ of $\mu $ as follows:
$$\tilde \mu_\alpha  (e^{is
})\,=\, \int_{\Gamma _\alpha (e^{is})}\,\frac{d\mu (z)}{1-\vert
z\vert },\quad s \in \mathbb R.$$
\begin{other}\label{th-q-car-gen} Let $\mu $ be a positive Borel measure on $\D $ and
$0<q<p<\infty $. Then $\mu $ is a $q$-Carleson measure for $H^p$ if
and only if $\tilde \mu _\alpha \in L^{\frac{p}{p-q}}(\partial \D )$
for some (equivalently, for all) $\alpha \in (0,\frac{\pi }{2})$.
\end{other}
\par\medskip A simple geometric argument shows that if the measure
$\mu $ is supported in $[0,1)$ then $\Gamma_\alpha (e^{is})\cap
[0,1)=[0,s_\alpha )$, where
$$1-s_\alpha \sim (\tan\alpha )\,s,\quad\text{as $s\to 0$}.$$
In particular, this implies the following.
\begin{other}\label{th-q-car-rad} Let $\mu $ be a positive Borel measure on $\D $
supported in $ [0,1)$, $0<q<p<\infty $. Then $\mu $ is a
$q$-Carleson measure for $H^p$ if and only if
\begin{equation}\label{eq-q-car-rad}
\int_0^1\left (\int_0^{1-s}\frac{d\mu (t)}{1-t}\right
)^{\frac{p}{p-q}}\,ds\,<\,\infty .\end{equation}
\end{other}

\medskip
\section{Proofs of the main results. Case $p\le 1$.}
\begin{Pf}{\,\em{Theorem \ref{Def:p<1}.}}
\par (i)\,\,$\Rightarrow
$\,\, (ii). Suppose that $\mu $ is an $1/p$-Carleson measure. Using
\cite[Theorem\,\@9.\,\@4]{D}, we see that there exists a positive
constant $C$ such that
\begin{equation}\label{car-1/p}\int_{[0,1)}\,\vert f(t)\vert\,d\mu
(t)\,\le\, C\Vert f\Vert _{H^p},\quad\text{for all\, $f\in
H^p$}.\end{equation} Take $f\in H^p$. Using (\ref{car-1/p}) we
obtain that
$$\sum_{n=0}^\infty \Bigl(\int
  _{[0,1)}t^{n}|f(t)| \,d\mu(t)\Bigr)|z|^n\,\le \,\frac{C\,\Vert f\Vert
  _{H^p}}{1-\vert z\vert },\quad z\in \D.$$ This implies that, for
  every $z\in \D$, the integral
  $$\int_{[0,1)}\,\frac{f(t)}{1-tz}\,d\mu
  (t)\,=\,\int_{[0,1)}\,f(t)\,\left (\sum_{n=0}^\infty \,t^nz^n\right
  )\,d\mu (t)$$ converges and that
  \begin{equation}\label{Imum}
I_\mu (f)(z)\,=\,\int_{[0,1)}\,\frac{f(t)}{1-tz}\,d\mu
  (t)\,=\,\sum_{n=0}^\infty \left (\int_{[0,1)}\,t^nf(t)\,d\mu
  (t)\right )z^n,\quad z\in \D .
  \end{equation}
  Thus $I_\mu (f)$ is a well defined analytic function in $\D $.

  \par (ii)\,\,$\Rightarrow
$\,\, (i). We claim that
\begin{equation}\label{fL1}\int_{[0,1)}\vert f(t)\vert \,d\mu(t)\,<\,\infty
,\quad\text{for all $f\in H^p$.}
\end{equation}
Indeed, take $f\in H^p$. Let $F$ be the function associated to $f$
by Proposition\,\@\ref{may-Hp}. Since $F\in H^p$, we have that the
integral $\int_{[0,1)}\frac{F(t)}{1-tz}\,d\mu (t)$ converges for all
$z\in \D$. Taking $z=0$ and bearing in mind that $0\le \vert
f(t)\vert \le F(t)$ ($t\in (0,1)$), we obtain that
$$\int_{[0,1)}\vert f(t)\vert \,d\mu(t)\,\le\,
\int_{[0,1)}F(t)\,d\mu(t)\, <\infty .$$ Thus (\ref{fL1}) holds.
\par For any $\beta \in [0,1)$ and $f\in H^p$ define
$$T_\beta (f)=f\cdot \chi_{\{0\le |z|<\beta\}}.$$
By (\ref{fL1}),  $T_\beta $ is a linear operator from $H^p$ into
$L^1(d\mu )$ and by the lemma in \cite[Section\,\@3.\,\@2]{D},
\begin{eqnarray*}\Vert T_\beta (f)\Vert_ {L^1(d\mu
)}\,=\,&\int_{[0,\beta )}\vert f(t)\vert \,d\mu (t)\,\le \left
[\sup_{\vert z\vert \le \beta }\vert f(z)\vert \right ]\cdot \mu
([0,\beta ) )\le C_\beta \Vert f\Vert _{H^p},\quad f\in H^p.
\end{eqnarray*}
 Thus, for every $\beta\in [0,1)$, $T_\beta $ is
a bounded linear operator from $H^p$ into $L^1(d\mu )$. Furthermore,
(\ref{fL1}) also implies that $$\sup_{0\le \beta <1}\Vert T_\beta
(f)\Vert_{L^1(d\mu )}\,\le \,\int _{[0,1)}\vert f(t)\vert \,d\mu
(t)=C_f<\infty ,\quad\text{for all $f\in H^p$}.$$ Then, by the
principle of uniform boundedness, we deduce that $\sup_{\beta \in
[0,1)}\Vert T_\beta \Vert <\infty $ which implies that the
identity operator is bounded from $H^p$ into $L^1(d\mu )$. Using
again \cite[Theorem\,\@9.\,\@4]{D} we obtain that $\mu $ is an
$1/p$-Carleson measure.
\par\medskip Assume now (i) (and (ii)), that
is, assume that $\mu $ is $1/p$-Carleson measure. Take $f\in H^p$,
$f(z)=\sum_{k=0}^\infty a_kz^k$ ($z\in \D $). By
Proposition\,\@\ref{moments-cond} and \cite[Theorem\,\@6.\,\@4]{D}
we have that there exists $C>0$ such that
$$\vert \mu_{n,k}\vert \,=\,\vert \mu_{n+k}\vert \,\le \,
\frac{C}{(k+1)^{1/p}}\,\,\,\text{and}\,\,\, \vert a_k\vert \le\,C
(k+1)^{(1-p)/p},\,\,\,\,\text{for all $n, k$}.$$ Then it follows
that, for every $n$,
\begin{equation*}\begin{split}\sum_{k=0}^\infty \vert \mu_{n,k}\vert \vert a_k\vert
\,\le \,& C\,\sum_{k=0}^\infty \frac{\vert a_k\vert
}{(k+1)^{1/p}}\,=\,C\,\sum_{k=0}^\infty \frac{\vert a_k\vert
^p\,\vert a_k\vert^{1-p} }{(k+1)^{1/p}} \\
\le \,& C\,\sum_{k=0}^\infty \frac{\vert a_k\vert
^p\,(k+1)^{(1-p)^2/p} }{(k+1)^{1/p}} \,=\,C\,\sum_{k=0}^\infty
(k+1)^{p-2}\vert a_k\vert ^p
\end{split}\end{equation*}
 and then
by a well known result of Hardy and Littlewood
(\cite[Theorem\,\@6.\,\@2]{D}) we deduce that $$\sum_{k=0}^\infty
\vert \mu_{n,k}\vert \vert a_k\vert \,\le \,C\,\Vert
f\Vert_{H^p}^p,\quad\text{for all $n$}.$$ This implies that
$\mathcal H_\mu $ is a well defined analytic function in $\D$ and
that
$$\int_{[0,1)}\,t^nf(t)\,d\mu (t)\,=\,\sum_{k=0}^\infty \mu
_{n,k}a_k,\quad\text{for all $n$},$$ bearing in mind (\ref{Imum}),
this gives that $\mathcal H_\mu (f)=I_\mu (f).$
\end{Pf}
\par\medskip

\begin{Pf}{\,\em{Theorem \ref{bound:p<1}.}}
\par Since $\mu $ is an $1/p$-Carleson measure, there exists $C>0$
such that (\ref{car-1/p}) holds. This implies that
\begin{equation}\label{cons-car-1/p}\int_0^{2\pi }\int_{[0,1)}\left \vert
\frac{f(t)\,g(e^{i\theta })}{1-re^{i\theta }t}\right\vert\,d\mu
(t)\,d\theta\,<\,\infty ,\quad{0\le r<1,\,\,\,f\in H^p,\,\,\,g\in
H^1.}\end{equation} Using Theorem\,\@\ref{Def:p<1},
(\ref{cons-car-1/p}) and Fubini's theorem, and the Cauchy's integral
representation of $H^1$-functions \cite[Theorem\,\@3.\,\@6]{D}, we
obtain
\begin{equation}\label{Hmu-du}\begin{split}&\int_0^{2\pi }\,\mathcal H_\mu (f)
(re^{i\theta })\,\overline {g(e^{i\theta })}\,d\theta = \int_0^{2\pi
}\left (\int_{[0,1)}\frac{f(t)\,d\mu (t)}{1-re^{i\theta
}t}\right)\overline {g(e^{i\theta })}\,d\theta \\ =\,&
\,\int_{[0,1)}f(t)\int_0^{2\pi }\,\frac{\overline {g(e^{i\theta
})}}{1-re^{i\theta }t}\,d\theta \,d\mu (t)=\int_{[0,1)}f(t)\overline
{g(rt)}d\mu (t),\,\,{0\le r<1,\,f\in H^p,\,g\in H^1.}
\end{split}\end{equation}
\par\medskip \par (i) Take $q\in (0,1)$.
Bearing in mind (\ref{Hmu-du}) and (\ref{car-1/p}) we deduce that
\begin{equation}\label{boundHmu}\left\vert
\int_0^{2\pi }\,\mathcal H_\mu (f)(re^{i\theta })\,\overline
{g(e^{i\theta })}\,d\theta \right\vert \,\le\,C\Vert
f\Vert_{H^p}\Vert g\Vert _{H^\infty },\,\,{0\le r<1,\,f\in
H^p,\,g\in H^\infty .}\end{equation} Now we recall
\cite[Theorem\,\@10]{DRS} that $B_q$ can be identified with the dual
of a certain subspace $X$ of $H^\infty $ under the pairing
$$<f, g>\,=\,\lim_{r\to 1}\frac{1}{2\pi }\int_0^{2\pi }\,f(re^{i\theta
})\,\overline{g(e^{i\theta })}\,d\theta ,\quad f\in B_q,\quad f\in
X.$$  This together with   (\ref{boundHmu}) gives that $\mathcal
H_\mu $ is a bounded operator from $H^p$ into $B_q$.

\par\medskip
(ii) We shall use Fefferman's duality theorem \cite{F, FS}, which
says that $(H^1)^\star\cong BMOA$ and $(VMOA)^\star \cong H^1$,
under the Cauchy pairing
\begin{equation}\label{eq:pai}
\langle f,g\rangle =\lim_{r\to 1^{-}}\frac{1}{2\pi}
\int_0^{2\pi}f(re^{i\theta})\ol{g(e^{i\theta})}\,d\theta,\quad f\in
H^1,\quad g\in BMOA\,\,(\text{resp.}\,\,VMOA),
\end{equation}
We mention \cite{Ba, Gar, G:BMOA},  as general references for the
spaces $BMOA$ and $VMOA$. In particular, Fefferman's duality theorem
can be found in \cite[Section\,\@7]{G:BMOA}.
\par Using the duality theorem and (\ref{cons-car-1/p}) it follows
 that $\mathcal H_\mu $ is a bounded operator from $H^p$ into
$H^1$ if and only there exists a positive constant $C$ such that
\begin{equation}\label{bound-p1}\left\vert \int_{[0,1)}f(t)\overline
{g(rt)}\,d\mu (t)\right\vert \,\le C\Vert f\Vert_{H^p}\Vert g\Vert
_{BMOA},\quad 0<r<1,\,\,\,f\in H^p,\,\,\,g\in VMOA.\end{equation}
\par Suppose that $\mathcal H_\mu $ is a bounded operator from $H^p$ to $H^1$. For $0<a,b<1$, let the functions $g_a$
and $f_b$ be defined by
\begin{equation}\begin{split}\label{eq:famfunctii}
g_a(z) &=\log\frac{2}{1-az}, \quad \quad f_b(z)
=\left(\frac{1-b^2}{(1-bz)^2}\right)^{1/p},\quad z\in \mathbb D.
\end{split}\end{equation}
\par A calculation shows that $\{g_a\}\subset VMOA$, $\{f_b\}\subset H^p
$, and
\begin{equation}\begin{split}\label{eq:hu1ii}
\sup_{a\in[0,1)}||g_a||_{BMOA}<\infty\quad\text{and}\quad
\sup_{b\in[0,1)}||f_b||_{H^p}<\infty.
\end{split}\end{equation}
\par Next,  taking $a=b\in [0,1)$ and $r\in [a,1)$, we obtain
\begin{equation*}\begin{split}
\left| \int_0^1f_a(t)\overline{g_a(rt)}\,d\mu(t)
 \right| &
\ge
\int_a^1\left(\frac{1-a^2}{(1-at)^2}\right)^{1/p}\log\frac{2}{1-rat}\,d\mu(t),
\\ & \ge C \frac{\log\frac{2}{1-a^2}}{(1-a^2)^{1/p}}\mu\left([a,1)\right),
\end{split}\end{equation*}
which, bearing in mind  \eqref{bound-p1} and \eqref{eq:hu1ii},
implies that $\mu$ is an $1$-logarithmic $\frac{1}{p}$-Carleson
measure.
\medskip\par
Reciprocally, suppose that $\mu$ is an $1$-logarithmic
$\frac{1}{p}$-Carleson measure. Let us see that $\mathcal H_\mu $ is
a bounded operator from $H^p$ to $H^1$. Using
 (\ref{bound-p1}),  it is enough to prove there exists $C>0$ such
 that
\begin{equation*}
\int_0^1|f(t)||g(rt)|\,d\mu(t)
 \le C||f||_{H^p}||g||_{BMOA},\quad\text{for all $r\in (0,1)$, $f\in H^p$, and $g\in VMOA$.}
\end{equation*}
By \cite[Theorem 9.4]{D}, this is equivalent to saying that, for
every $r\in (0,1)$ and every $g\in VMOA$, the measure $\vert
g(rz)\vert \,d\mu (z)$ is an $1/p$-Carleson measure with constant
bounded by $C\Vert g\Vert _{BMOA}$. Using Lemma \ref{le:zh} this can
be written as
\begin{equation}\label{eq:hu2ii}
\sup_{a\in
\D}\int_{\D}\left(\frac{1-|a|^2}{|1-\bar{a}z|^2}\right)^{1/p}|g(rz)|\,d\mu(z)\le
C||g||_{BMOA},\quad\text{$0<r<1$,\,\, $g\in VMOA$}.
\end{equation}
\par So take $r\in(0,1)$, $a\in\D$ and $g\in VMOA$. We have
\begin{equation*}\begin{split}
&\int_{\D}\left(\frac{1-|a|^2}{|1-\bar{a}z|^2}\right)^{1/p}|g(rz)|\,d\mu(z)
\\ & \le |g(ra)|\int_{\D}\left(\frac{1-|a|^2}{|1-\bar{a}z|^2}\right)^{1/p}\,d\mu(z)
+\int_{\D}\left(\frac{1-|a|^2}{|1-\bar{a}z|^2}\right)^{1/p}|g(rz)-g(ra)|\,d\mu(z)
\\
&=I_ 1(r,a)+I_ 2(r,a).
\end{split}\end{equation*}
Bearing in mind that any function $g$ in the Bloch space
$\mathcal{B}$ (see \cite{ACP}) satisfies the growth condition
\begin{equation}\label{blochgrowth}
|g(z)|\leq 2\|g\|_{\mathcal{B}}\,\log \frac{2}{1-|z|},\quad\text{for
all $z\in\D$},
\end{equation}
and $BMOA\subset \mathcal{B}$ \cite[Theorem\,\@5.\,\@1]{G:BMOA}), by
Lemma \ref{le:zh} we have  that
\begin{equation}\begin{split}\label{eq:hu4ii}
I_ 1(r,a)&\leq C||g||_{BMOA}\log
\frac{2}{1-|a|}\,\int_{\D}\left(\frac{1-|a|^2}{|1-\bar{a}z|^2}\right)^{1/p}\,d\mu(z)
\\ & \le CK_{1,\frac{1}{p}}(\mu)||g||_{BMOA}<\infty .
\end{split}\end{equation}
\par
Now, since $\mu $ is an $1/p$-Carleson measure, Lemma \ref{le:zh}
yields

\begin{equation*}\begin{split}
I_ 2(r,a)  \leq &\,CK_{\frac{1}{p}}(\mu) \left\Vert \left
(\frac{1-\vert a\vert ^2}{(1-\overline a\,z)^2}\right )^{1/p}\left
[g(rz)-g(ra)\right ]\right\Vert _{H^p}
\\ =& \, CK_{\frac{1}{p}}(\mu) \left (
\int_{0}^{2\pi}\frac{1-|a|^2}{|1-\bar{a}e^{i\theta}|^2}|g_r(e^{i\theta})-g_r(a)|^p\,d\theta
\right )^{1/p}\\ \le & \,
CK_{\frac{1}{p}}(\mu)\int_{0}^{2\pi}\frac{1-|a|^2}{|1-\bar{a}e^{i\theta}|^2}|g_r(e^{i\theta})-g_r(a)|\,d\theta
,\end{split}\end{equation*} where, $g_r(z)=g(rz)$ ($z\in\D $). Now,
using the conformal invariance of $BMOA$
(\cite[Theorem\,\@3.\,\@1]{G:BMOA})) and the fact that $BMOA$ is
closed under subordination \cite[Theorem\,\@10.\,\@3]{G:BMOA}, we
obtain that
$$\int_{0}^{2\pi}\frac{1-|a|^2}{|1-\bar{a}e^{i\theta}|^2}|g_r(e^{i\theta})-g_r(a)|\,d\theta
\le C\Vert g\Vert _{BMOA}$$ and then it follows that $I_ 2(r,a)\le
CK_{\frac{1}{p}}(\mu)\Vert g\Vert _{BMOA}$. This and
(\ref{eq:hu4ii}) give (\ref{eq:hu2ii}), finishing the proof of
part\,\@(ii).

\par\medskip \par (iii)  Using
(\ref{Hmu-du}), the duality theorem for $H^q$
\cite[Section\,\@7.\,\@2]{D} and arguing as in the proof of
part\,\@(ii), we can assert that $\mathcal H_\mu $ is a bounded
operator from $H^p$ to $H^q$ if and only if there exists a positive
constant $C$ such that
\begin{equation}\label{Hmuboundedpq>1++}
\left \vert \int_{[0,1)}\,f(t)\,\overline {g(t)}\,d\mu (t)\right
\vert \,\le \,C\,\Vert f\Vert _{H^p}\,\Vert g\Vert
_{H^{q^\prime}},\quad f\in H^p,\quad g\in H^{q^\prime
}.\end{equation} Now, by Proposition\,\@\ref{may-Hp}, it follows
that (\ref{Hmuboundedpq>1++}) is equivalent to
\begin{equation}\label{Hmuboundedpq>1--}
\int_{[0,1)}\,\vert f(t)\vert \,\vert g(t)\vert \,d\mu (t)\,\le
\,C\,\Vert f\Vert _{H^p}\,\Vert g\Vert _{H^{q^\prime}},\quad f\in
H^p,\quad g\in H^{q^\prime },\end{equation} and, by
Lemma\,\@\ref{le:zh}, this is the same as saying the following:
\par For every $g\in H^{q^\prime }$, the measure $\mu_g$ supported
on $[0,1)$ and defined by $d\mu_g(z)=\vert g(z)\vert \,d\mu (z)$ is
a $1/p$-Carleson measure with $K_{\frac{1}{p}}(\mu _g)\le C\,\Vert
g\Vert _{H^{q^\prime }}$, that is,
\begin{equation}\label{Hmuboundedpq>1}\sup _{a\in \D
}\,\int_{[0,1)}\,\left (\frac{1-\vert a\vert ^2}{\vert 1-\overline
at\vert^2}\right )^{1/p}\,\vert g(t)\vert \,d\mu (t)\,\le \, C\Vert
g\Vert _{H^{q^\prime}},\quad g\in H^{q^\prime }.\end{equation}

\par Suppose that
$\mathcal H_\mu $ is a bounded operator from $H^p$ to $H^q$. Then
(\ref{Hmuboundedpq>1}) holds. For $a\in \D$, take
$$g_a(z)=\left (\frac{1-\vert a\vert ^2}{(1-\overline az)^2}\right
)^{1/q^\prime },\quad z\in \D.$$ Since
 $\sup_{a\in \D} \Vert g_a\Vert_{H^{q^\prime
}}<\infty $, (\ref{Hmuboundedpq>1}) implies that $$\sup _{a\in
\D }\,\int_{[0,1)}\,\left (\frac{1-\vert a\vert ^2}{\vert
1-\overline at\vert^2}\right )^{\frac{1}{p}+\frac{1}{q^\prime
}}\,d\mu (t)<\infty ,$$  that is,
$\mu $ is a $\frac{1}{p}+\frac{1}{q^\prime }$-Carleson measure, by Lemma\,\@\ref{le:zh}.

\medskip\par Suppose now that $\mu $ is an $\frac{1}{p}+\frac{1}{q^\prime }$-Carleson
measure. Set $s=1+\frac{p}{q^\prime }$. The  conjugate exponent
of $s$ is $s^\prime =1+\frac{q^\prime }{p}$ and
$\frac{1}{p}+\frac{1}{q^\prime }\,=\,\frac{s}{p}\,=\,\frac{s^\prime }{q^\prime }.$ Then, by
\cite[Theorem~9.\@4]{D},  $H^p$ is continuously embedded
in $L^s(d\mu )$ and $H^{q^\prime }$ is continuously embedded in
$L^{s^\prime }(d\mu )$, that is, \begin{equation}\label{s}\left
(\int_{[0,1)}\vert f(t)\vert ^s\,d\mu (s)\right )^{1/s}\,\le
\,C\,\Vert f\Vert _{H^p},\quad f\in H^p,\end{equation} and
\begin{equation}\label{sprime}\left (\int_{[0,1)}\vert g(t)\vert
^{s^\prime }\,d\mu (s)\right )^{1/{s^\prime }}\,\le \,C\,\Vert
g\Vert _{H^{q^\prime }},\quad g\in H^{q^\prime }.\end{equation}
Using H\"{o}lder's inequality with exponents $s$ and $s^\prime $,
(\ref{s}) and (\ref{sprime}), we obtain
\begin{equation*}\begin{split} \int_{[0,1)}\vert f(t)\vert \,\vert g(t)\vert\,d\mu
(t)\,\le & \left (\int_{[0,1)}\vert f(t)\vert ^s\,d\mu (s)\right
)^{1/s}\,\left (\int_{[0,1)}\vert g(t)\vert ^{s^\prime }\,d\mu
(s)\right )^{1/{s^\prime }}\\ \le & C\,\Vert f\Vert _{H^p}\,\Vert
g\Vert _{H^{q^\prime }},\quad f\in H^p,\,\, g\in H^{q^\prime
}.\end{split}\end{equation*} Hence, (\ref{Hmuboundedpq>1--}) holds
and then it follows that $\mathcal H_\mu $ is a bounded operator
from $H^p$ to $H^q$.
\end{Pf}
\medskip

\section{Proofs of the main results. Case $p>1$.}
\begin{Pf}{\,\em{Theorem \ref{Def:p>1}\,}}
(i). Since $\mu $ is an $1$-Carleson measure for $H^p$,
(\ref{car-1/p}) holds for a certain $C>0$.
 Then the argument used in the proof of the
implication (i)\,$\Rightarrow$\, (ii) in Theorem\,\@\ref{Def:p<1}
gives that, for every $f\in H^p$, $I_\mu (f)$ is a well defined
analytic function in $\D $ and
\begin{equation}\label{Imu-power}
I_\mu (f)(z)\,=\,\sum_{n=0}^\infty \left
(\int_{[0,1)}\,t^nf(t)\,d\mu
  (t)\right )z^n,\quad z\in \D .\end{equation}
  \par The reverse implication can be proved just as (ii)\,$\Rightarrow$\, (i) in
  Theorem\,\@\ref{Def:p<1}.
  \par The fact that $\mu $ being an $1$-Carleson measure for $H^p$
  is equivalent to (\ref{int-car-1p}) follows from
  Theorem\,\@\ref{th-q-car-rad}.
\par
(ii). Take $f\in
H^p$, $f(z)=\sum_{k=0}^\infty a_kz^k$ ($z\in \D $). Set
$$S_n(f)(z)=\sum_{k=0}^n a_kz^k,\quad R_n(f)(z)=\sum_{k=n+1}^\infty
a_kz^k,\quad z\in \D ,\quad n= 0, 1, 2, \dots .$$ Whenever $0\le
N<M$ and $n\ge 0$, we  have
\begin{equation*}\label{cauchy}\begin{split}
\left \vert \sum_{k=N+1}^M\,\mu_{n,k}a_k\right\vert \,=\,& \left
\vert \int_{[0,1)}t^n\,\left (\sum_{k=N+1}^M\,a_k\,t^k\right )\,d\mu
(t)\right\vert \\ =\,& \left \vert \int_{[0,1)}t^n\,\left
[S_M(f)(t)-S_N(f)(t)\right ]\,d\mu (t)\right\vert \\ \le \,&
\int_{[0,1)}\left \vert S_M(f)(t)-S_N(f)(t)\right \vert \,d\mu (t).
\end{split}\end{equation*}
Using this, the fact that $\mu $ is an $1$-Carleson measure for
$H^p$, and the Riesz projection theorem, we deduce that
$$\sum_{k=N+1}^M\,\mu_{n,k}a_k\to 0,\quad \text{as $N, M\to \infty
$}$$ for all $n$. This gives that the series $\sum_{k=0}^\infty
\,\mu_{n,k}a_k$ converges for all $n$.
\par For $n, N\ge 0$, we have
\begin{equation*}\begin{split}
\left\vert \int_{[0,1)}\,t^n\,f(t)\,d\mu
(t)\,-\,\sum_{k=0}^N\,\mu_{n,k}\,a_k\right\vert \,= & \, \left\vert
\int_{[0,1)}\,t^n\,f(t)\,d\mu (t)\,-\,\int_{[0,1)}\,t^n\,\left
(\sum_{k=0}^N\,a_k\,t^t\right )\,d\mu (t)\right\vert \\ = &
 \left\vert
\int_{[0,1)}\, t^n\,R_{N}(f)(t)\,d\mu (t)\right \vert \\ \le &
C\Vert R_N(f)\Vert _{H^p}.
\end{split}\end{equation*}
Since $1<p<\infty $, $\Vert R_N(f)\Vert _{H^p}\to 0$, as $N\to\infty
$, and then it follows that
$$\sum_{k=0}^\infty \,\mu_{n,k}a_k\,=\,\int_{[0,1)}\,t^n\,f(t)\,d\mu
(t),\quad \text{for all $n$},$$ which together with (\ref{car-1/p})
implies that $\mathcal H_\mu (f)$ is a well defined analytic
function in $\D$
 and, by (\ref{Imu-power}),
$\mathcal H_\mu (f)=I_\mu (f)$.
\end{Pf}
\par\bigskip Let us turn to prove Theorem\,\@\ref{bound:p>1}. In
view of Theorem\,\@\ref{Def:p>1}, $\mathcal H_\mu$ coincides with
$I_\mu $ on $H^p$. This fact will be used repeatedly in the
following.
\par Recall that (\ref{int-car-1p}) implies that $\mu $ is
an $1$-Carleson measure for $H^p$, that is, we have
$$\int_{[0,1)}\,\vert f(t)\vert\,d\mu (t)\,\le \,C\Vert f\Vert
_{H^p},\quad f\in H^p.$$ Then arguing as in the proof of
Theorem\,\@\ref{bound:p<1}, we obtain

\begin{equation}\label{Hmu-du-p>1}\begin{split}&\int_0^{2\pi }\,\h_\mu (f)
(re^{i\theta })\,\overline {g(e^{i\theta })}\,d\theta = \int_0^{2\pi
}\left (\int_{[0,1)}\frac{f(t)\,d\mu (t)}{1-re^{i\theta
}t}\right)\overline {g(e^{i\theta })}\,d\theta \\ =\,&
\,\int_{[0,1)}f(t)\int_0^{2\pi }\,\frac{\overline {g(e^{i\theta
})}}{1-re^{i\theta }t}\,d\theta \,d\mu (t)=\int_{[0,1)}f(t)\overline
{g(rt)}d\mu (t),\,\,{0\le r<1,\,f\in H^p,\,g\in H^1.}
\end{split}\end{equation}

\par\medskip Once (\ref{Hmu-du-p>1}) is established, (i) can be proved with the argument used in the proof of part\,\@(iii) of
Theorem\,\@\ref{bound:p<1}.
\par\medskip
 Part\,\@(ii) of Theorem\,\@
\ref{bound:p>1} is a byproduct of the following result.

\begin{proposition}\label{acotacionhpmenorq}
Asumme that $1<q<p<\infty$ and  let $\mu$ be a positive Borel
measure on $[0,1)$ satisfying (\ref{int-car-1p}). Then, the
following conditions are equivalent:\begin{itemize} \item[(a)]\,
$\h_\mu $ is a bounded operator from  $H^p$ to $H^q$. \item[(b)]\,
\, $\h_\mu $ is a bounded operator from $H^{\frac{2pq'}{p+q'}}$ to
$H^{\left(\frac{2pq'}{p+q'}\right)'}$.
\item[(c)]\,$H^{\frac{2pq'}{p+q'}}$ is continuously contained in $L^{2}(\mu)$.
\item[(d)]\, The function defined by
$s\mapsto\int_{0}^{1-s}\frac{d\mu(t)}{1-t}$ ($s\in [0,1)$) belongs
to $L^{\left(\frac{pq'}{p+q'}\right)'}([0,1))$.
\end{itemize}
\end{proposition}
\begin{proof}
\par (a)\, $\Rightarrow $ \,(b).
Using duality as above, we see that (a) is equivalent to
\begin{equation}\label{Imub}
\left \vert \int_{[0,1)}\,f(t)\,\overline {g(t)}\,d\mu (t)\right
\vert \,\lesssim\,\Vert f\Vert _{H^p}\,\Vert g\Vert
_{H^{q^\prime}},\quad f\in H^p,\quad g\in H^{q^\prime
}.\end{equation} Take $f\in H^p$ and $g\in H^{q'}$, and let $F\in
H^p$ and $G\in H^{q'}$ be the functions associated to $f$ and $g$ by
Proposition\,\@\ref{may-Hp}, respectively. Using (\ref{Imub}) and
H\"{o}lder's inequality, we obtain
\begin{equation}\label{eqj1}\begin{split}
 \int_0^1|f(t)||g(t)|\,d\mu(t) & \lesssim \int_{0}^1 F(t)\,{G(t)}\,d\mu(t)
\,= \left|  \int_{0}^1 F(t)\,\overline {G(t)}\,d\mu(t) \right|
\\ & \lesssim||F||_{H^p}||G||_{H^{q'}}\lesssim ||f||_{H^p}||g||_{H^{q'}}.
\end{split}\end{equation}
Take now $\phi\in H^{\frac{2pq'}{p+q'}}$. By the outer-inner
factorization \cite[Chapter $2$]{D}, $\phi=\Phi\cdot I$ where $I$ is
an inner function and $\Phi\in H^{\frac{2pq'}{p+q'}}$ is free from
zeros and $||\Phi||_{ H^{\frac{2pq'}{p+q'}}}=||\phi||_{
H^{\frac{2pq'}{p+q'}}}$. Now let us consider the analytic functions
$f=\Phi^{\frac{2q'}{p+q'}}$ and $g=\Phi^{\frac{2p}{p+q'}}$. We have
$$f=\Phi^{\frac{2q'}{p+q'}}\in H^p,\quad\text{with
$||f||_{H^p}=||\Phi||^{\frac{2q'}{p+q'}}_{H^{\frac{2pq'}{p+q'}}}$}$$
and
$$g=\Phi^{\frac{2p}{p+q'}}\in H^{q'},\quad\text{with $||g||_{H^{q'}}=||\Phi||^{\frac{2p}{p+q'}}_{H^{\frac{2pq'}{p+q'}}}$}.$$
Bearing in mind (\ref{eqj1}), it follows that
\begin{equation*}\begin{split}
\int_0^1|\phi(t)|^2\,d\mu(t)& \le  \int_0^1|\Phi(t)|^2\,d\mu(t)
 \\ & = \int_0^1|f(t)||g(t)|\,d\mu(t)
\\ & \lesssim ||f||_{H^p}||g||_{H^{q'}}=||\Phi||^2_{ H^{\frac{2pq'}{p+q'}}}=||\phi||^2_{ H^{\frac{2pq'}{p+q'}}},
\end{split}\end{equation*}
which gives (b).
\par (b)\, $\Rightarrow $ \,(c).
Since $p>q>1$, $\frac{pq'}{p+q'}>1$, by duality, as above, (b) is
equivalent to
$$\left|\int_0^1f(t)\overline{g(t)}\,d\mu(t)\right|\,\le C\,||f||_{H^{\frac{2pq'}{p+q'}}}||g||_{H^{\frac{2pq'}{p+q'}}},\quad f,g\in H^{\frac{2pq'}{p+q'}}.$$
Taking $f=g$ we obtain
$$\int_0^1\,\vert f(t)\vert ^2\,d\mu(t)\,\le
C\,||f||_{H^{\frac{2pq'}{p+q'}}}^2.$$ This is (c).
\medskip\par Theorem\,\@\ref{th-q-car-rad} gives that (c) and (d) are
equivalent.
\medskip
\par (d)\, $\Rightarrow $ \,(a).  Using again
Theorem\,\@\ref{th-q-car-rad} we have that $H^p$ is continuously
contained in  $L^\frac{p+q'}{q'}(d\mu)$ and $H^{q'}$ is continuously
contained in $L^\frac{p+q'}{p}(d\mu)$, which together with
H\"{o}lder's inequality gives
\begin{equation*}\begin{split}
\int_0^1|f(t)||g(t)|\,d\mu(t)\le &
\left(\int_0^1|f(t)|^{\frac{p+q'}{q'}}\,d\mu(t)\right)^{\frac{q'}{p+q'}}
\left(\int_0^1|g(t)|^{\frac{p+q'}{p}}\,d\mu(t)\right)^{\frac{p}{p+q'}}
\\ & \le C ||f||_{H^p}||g||_{H^{q'}},\quad f\in H^p,\quad g\in H^q,
\end{split}\end{equation*}
and this is equivalent to (a).
\end{proof}

\begin{Pf}{\,\em{Theorem \ref{bound:p>1}\,(iii).}} Just as in the
proof of Theorem\,\@\ref{bound:p<1}\,\@(ii),  $\h_\mu $ is a bounded
operator from $H^p$ into $H^1$ if and only there exists a positive
constant $C$ such that
\begin{equation}\label{bound-p1-p>1}\left\vert \int_{[0,1)}f(t)\overline
{g(rt)}\,d\mu (t)\right\vert \,\le C\Vert f\Vert_{H^p}\Vert g\Vert
_{BMOA},\quad 0<r<1,\,\,\,f\in H^p,\,\,\,g\in VMOA.\end{equation}
Let $\nu $ be the measure on $[0,1)$ defined by $d\nu
(t)=\log\frac{1}{1-t}\,d\mu (t)$, by
Theorem\,\@\ref{th-q-car-rad} it follows that  the function
\,\,$s\mapsto \int_0^{1-s}\,\frac{\log\frac{1}{1-t}d\mu
(t)}{1-t}$\,\, $(s\in [0,1))$\, belongs to $L^{p^\prime }([0,1))$ if
and only if the measure $\nu $ is an $1$-Carleson measure for $H^p$.
\par Consequently, we have to prove that
$$\text{(\ref{bound-p1-p>1})\,\,\,$\Leftrightarrow $\,\,\,$\nu $ is an $1$-Carleson measure for
$H^p$.}$$
\par Suppose that (\ref{bound-p1-p>1}) holds. For $0<\rho <1$, let
$g_\rho $ be the function defined by $g_\rho (z)=\log\frac{1}{1-\rho
z}$ ($z\in \D $), then
\begin{equation*}\text{$g_\rho \in VMOA$, for all $\rho \in
(0,1)$,\,\,\,\, and \,\, $\sup_{0<\rho <1}\Vert g_\rho \Vert
_{BMOA}=A<\infty $ }.\end{equation*} On the other hand, if $f\in
H^p$, $0<r<1$, and
 $F$ is the function associated to $f$ by
Proposition\,\@\ref{may-Hp}, it follows that
\begin{equation*}\begin{split}
\int_{[0,1)}\vert f(t)\vert \,\log\frac{1}{1-\rho rt}\,d\mu (t)\,\le
\,\int_{[0,1)}\,F(t)\,\overline{g_\rho (rt)}\,d\mu (t)\,\le
C\,A\,\Vert F\Vert _{H^p}\,= C\,A\,\Vert f\Vert
_{H^p},\end{split}\end{equation*} for every $\rho\in (0,1)$. Letting
$r$ and  $\rho $ tend to $1$, we obtain
\begin{equation*}\begin{split}
\int_{[0,1)}\vert f(t)\vert \,\log\frac{1}{1-t}\,d\mu (t)\,\le
C\,A\,\Vert f\Vert _{H^p}.\end{split}\end{equation*} Thus $\nu $ is
an $1$-Carleson measure for $H^p$.

\par Conversely, assume that $\nu $ is
an $1$-Carleson measure for $H^p$. Take $r\in (0,1),\,f\in H^p,$\,
and\,$\,g\in VMOA$.
 Using (\ref{blochgrowth}), we obtain $$
\left\vert \int_{[0,1)}f(t)\overline {g(rt)}\,d\mu (t)\right\vert
\,\le \,C\,\Vert g\Vert_{BMOA}\int_{[0,1)}\vert f(t)\vert
\log\frac{2}{1-t}\,d\mu (t)\,\le\,C\,\Vert g\Vert_{BMOA}\,\Vert
f\Vert_{H^p}.$$
\end{Pf}
\par\medskip
\begin{Pf}{\,\em{Theorem \ref{bound:p>1}\,(iv).}}
Assume that the function defined by
$\,s\mapsto\,\int_0^{1-s}\frac{d\mu (t)}{1-t}$ ($s\in [0,1)$)
belongs to $L^{p^\prime }([0,1)$. By Theorem\,\@\ref{th-q-car-rad},
this implies that $H^p$ is continuously contained in $L^1(d\mu )$.
From now on, the proof is analogous to
 the proof
of Theorem\,\@\ref{bound:p<1}\,\@(i).
\end{Pf}\par\bigskip

\section{Compactness.}\label{sect-compact}
\par The next theorem gathers our main results concerning the study of the compactness of  $\h_\mu$ on
Hardy spaces.
\begin{theorem}\label{compactness} Let $\mu $ be a positive Borel
measure on $[0,1)$. \begin{itemize}
\item[(i)] If \, $0<p\le 1$\, and  $\mu $ is a $1/p$-Carleson measure, then  $\h_\mu $ is a compact
operator from $H^p$ to $H^1$ if and only if $\mu $ is a vanishing
$1$-logarithmic $1/p$-Carleson measure.
\item[(ii)] If\, $0<p\le 1< q$ and and  $\mu $ is a $1/p$-Carleson measure, then $\h_\mu $ is a compact
operator from $H^p$ to $H^q$ if and only if $\mu $ is a vanishing
$\frac{1}{p}+\frac{1}{q^\prime }$-Carleson measure.
\item[(iii)] If\, $1<p< q$  and  $\mu $ satisfies (\ref{int-car-1p}), then $\h_\mu $ is a compact
operator from $H^p$ to $H^q$ if and only if $\mu $ is a vanishing
$\frac{1}{p}+\frac{1}{q^\prime }$-Carleson measure.
\item[(iv)] If $1<p<\infty $,  $\mu $ satisfies (\ref{int-car-1p}) and $1\le q<p$, then $\mathcal H_\mu $ is a compact
operator from $H^p$ to $H^q$ if and only if it is a bounded operator
from $H^p$ to $H^q$.
\end{itemize}\end{theorem}
\par\medskip The following lemma will be used in the proof of cases (i), (ii) and (iii).
\begin{lemma}\label{lemafb} Suppose that $0<p<\infty $ and let $\mu $ be a positive Borel
measure on $[0,1)$ which is an $1$-logarithmic $1/p$-Carleson
measure. Let $f_b$, $(0\le b<1)$, be defined as in
(\ref{eq:famfunctii}). Then
\begin{equation}\label{int-fb-zero}\lim_{b\to
1^-}\int_{[0,1)}\,f_b(t)\,d\mu (t)\,=\,0.\end{equation}
\end{lemma}
\begin{pf} For $0\le t<1$, set $F(t)\,=\,\mu \left ([0,t)\right
)\,-\,\mu \left ([0,1)\right )\,=\,-\,\mu \left ([t,1)\right) $.
Integrating by parts and using the fact that $\mu $ is an
$1$-logarithmic $1/p$-Carleson measure, we obtain
\begin{equation}\label{parts-fb}\int_{[0,1)}\,f_b(t)\,d\mu (t)\,=\,
f_b(0)\,\mu \left ([0,1)\right )\,+\,\int_0^1\,f_b^\prime (t)\,\mu
\left ([t,1)\right )\,dt.\end{equation} Using that $\mu $ is an
$1$-logarithmic $1/p$-Carleson measure and the fact that $bt<b$ and
$bt<t$ ($0<\,b,t\,<1$), we deduce
\begin{equation*}\begin{split}\,&\int_0^1\,f_b^\prime (t)\,\mu
\left ([t,1)\right )\,dt\le\,
C\int_0^1\frac{(1-b)^{1/p}(1-t)^{1/p}}{(1-bt)^{\frac{2}{p}+1}\log\frac{e}{1-t}}\,dt\\
= &
C\int_0^b\frac{(1-b)^{1/p}(1-t)^{1/p}}{(1-bt)^{\frac{2}{p}+1}\log\frac{e}{1-t}}\,dt\,+\,
C\int_b^1\frac{(1-b)^{1/p}(1-t)^{1/p}}{(1-bt)^{\frac{2}{p}+1}\log\frac{e}{1-t}}\,dt
\\
\le &
C\,(1-b)^{1/p}\,\int_0^b\frac{dt}{(1-t)^{\frac{1}{p}+1}\log\frac{e}{1-t}}\,+
\,\frac{C}{(1-b)^{\frac{1}{p}+1}}\int_b^1\frac{(1-t)^{1/p}}{\log\frac{e}{1-t}}\,dt
\\
= & I(b)\,+\,II(b).
\end{split}\end{equation*}
Now, it is a simple calculus exercise to show that $I(b)$ and
$II(b)$ tend to $0$, as $b\to 1$. Using this, the fact that\,
$\lim_{b\to 1}f_b(0)\to 0$, and (\ref{parts-fb}), we deduce
(\ref{int-fb-zero}).
\end{pf}
\par\medskip
\begin{Pf}{\,\em{Theorem \ref{compactness}\,\@}}
(i).\, Suppose that $\mathcal H_\mu $ is a compact operator from
$H^p$ to $H^1$. Let $f_b$, $(0\le b<1)$, be defined as in
(\ref{eq:famfunctii}). Let $\{ b_n\} \subset (0,1)$ be any sequence
with $b_n\to 1$ and  such that the sequence $\{ \mathcal H_\mu
(f_{b_n})\} $ converges in $H^1$ (such a sequence exists because
$\sup _{0<b<1}\Vert f_b\Vert _{H^p}<\infty $ and $\mathcal H_\mu $
is compact) and let $g$ be the limit (in $H^1$) of $\{ \mathcal
H_\mu (f_{b_n})\} $. Then $\mathcal H_\mu (f_{b_n})\to g$, uniformly
on compact subsets of $\D $. Now, by Theorem \ref{Def:p<1}, we have
$$0\le H_\mu (f_{b_n})(r)=\int_{[0,1)}\frac{f_{b_n}(t)}{1-rt}\,d\mu
(t)\le \frac{1}{1-r}\int_{[0,1)}\,f_{b_n}(t)\,d\mu (t),\quad
0<r<1.$$ Since $\mathcal H_\mu $ is continuous from $H^p$ to $H^1$,
$\mu $ is an  $1$-logarithmic $1/p$-Carleson measure. Then, by
Lemma\,\@\ref{lemafb}, it follows that $g(r)=0$ for all $r\in
(0,1)$. Hence, $g\equiv 0$. In this way we have proved that
$$\mathcal H_\mu (f_b)\to 0, \quad\text{as $b\to 1$,\,\, in
$H^1$.}$$ Arguing as in proof of
 the boundedness
(Theorem\,\@\ref{bound:p<1}\,\@(ii)), this yields
$$\lim_{b\to 1^-}\frac{\mu \left ([b,1)\right
)\log\frac{e}{1-b}}{(1-b)^{1/p}}\,=\,0,$$ which is equivalent to
saying that $\mu $ is a vanishing $1$-logarithmic $1/p$-Carleson
measure.
\par\medskip
Suppose now that $\mu $ is a vanishing $1$-logarithmic
$1/p$-Carleson measure.  Let $\{ f_n\} _{n=1}^\infty $ be a sequence
of functions in $H^p$ with $\sup\Vert f_n\Vert_{H^p}<\infty $ and
such that $f_n\to 0$, uniformly on compact subsets of $\D $. For
$0<r<1$, let us write
$$d\mu_r(t)=\chi _{r<\vert z\vert <1}(t)\,d\mu (t).$$
Since $\mu $ is a vanishing $1$-logartihmic $1/p$-Carleson measure,
$\lim_{r\to 1}K_{1,\frac{1}{p}}(\mu _r)=0$. This together with the
fact that $f_n\to 0$, uniformly on compact subsets of $\D $
gives that
$$\lim_{n\to\infty }\int_0^1\,\vert f_n(t)\vert\,\vert
g(t)\vert\,d\mu (t)\,=\,0\quad \text{for all $g\in VMOA$.}$$ Using
the duality relation $(VMOA)^\star \cong H^1$ as in the proof of
Theorem\,\@\ref{bound:p<1}, this implies that $\mathcal H_\mu
(f_n)\to 0$, in $H^1$. So, $\h_\mu:\, H^p\to H^1$ is compact.

\par Parts (ii) and (iii)\, can be proved
similarly to  the preceding one. We shall omit the details. Let us
simply remark, for the necessity,
 that if $\mu $ is a vanishing $\frac{1}{p}+\frac{1}{q^\prime
}$-Carleson measure, then it is an $1$-logarithmic $1/p$-Carleson
measure and then we can use Lemma\,\@\ref{lemafb}.
\end{Pf}
\par\medskip
\begin{Pf}{\,\em{Theorem \ref{compactness}\,\@(iv).}}
Suppose that $1<p<\infty $ and $0<q<p$. Looking at the proof of
parts (ii) and (iii) of Theorem\,\@\ref{bound:p>1}, we see
$\mathcal H_\mu $ applies $H^p$ into $H^q$ if and only if:
\begin{itemize}\item $H^\frac{2pq^\prime}{p+q^\prime }$ is continuously
embedded in $L^2(d\mu )$, in the case $1<q<p$.
\item  $H^p$ is continuously
embedded in $L^1(d\nu )$, where $d\nu (t)=\log\frac{1}{1-t}\,d\mu
(t)$, in the case $q=1<p<\infty $.
\end{itemize}
Now, using the results in Section\,\@3 of \cite{Bl-Ja}, we see that
when any of these embeddings exists as a continuous operator, then
it is compact. Then arguments similar to those used in the
boundedness case can be used to yield the compactness. We omit the
details.
\end{Pf}
\section{Schatten classes.}\label{Schatten}

\medskip\par Before presenting the proof of Theorem \ref{th:Schatten} let us recall some definitions
which connect the operator $\h_\mu$ with the classical theory of
Hankel operators.  \par Given $\varphi
(\xi)\sim\sum_{n=-\infty}^{+\infty}\widehat{\varphi}(n)\xi^n\in
L^2(\T)$, the associated (big) Hankel operator
\par\noindent $H_\varphi: H^2\to H^2_{-}$
 (see \cite[p. $6$]{Pell}) is formally defined
as
$$H_\varphi(f)=I-P(\varphi f)$$
where $P$ is the Riesz projection.

\par Moreover, if $\mu$ is a classical Carleson measure, Nehari's Theorem implies that
(see  \cite[p.\,\@3 and p.\,\@42]{Pell}) there exists
$\varphi_\mu\in L^\infty (\T)$ with
$\mu_{n+1}=\widehat{\varphi_\mu}(-n)$, so
$$\h_\mu(f)(z)-\h_\mu(f)(0)=H_{\varphi_\mu}(f)(\bar{z}).$$
In particular, $\h_\mu$ is bounded on $H^2$ if and only if
$H_{\varphi_\mu}: H^2\to H^2_{-}$ is bounded, that is, if  and only
if \,$\mu$ is a Carleson measure. Finally let us observe that,
$$(I-P)(\varphi_\mu)(\bar{z})=\sum_{n=1}^\infty \mu_{n+1}z^n$$

\begin{Pf}{\em{Theorem \ref{th:Schatten}.}}
It follows from the above observation, \cite[p. $240$,
Corollary\,\@2.\,\@2]{Pell} and \cite[Appendix $2.6$]{Pell} that
$\mathcal H_\mu \in \mathcal S_p(H^2)$ if and only if
$h_\mu(z)=\sum_{n=1}^\infty \mu_{n+1}z^{n}\in B^p$.  Bearing in mind
\cite[Theorem $2.1$]{MP} (see also \cite[p. $120$,
$7.5.8$]{Pabook}),
 \cite[p. $120$, $7.3.5$]{Pabook}, the fact that $\{\mu_n\}$ decreases to zero  and \cite[Lemma $3.4$]{LNP}, we deduce that
\begin{equation*}\begin{split}
||h_\mu||^p_{B^p} & \asymp \sum_{n=0}^\infty
2^{-n(p-1)}\left\|\sum_{k=2^n}^{2^{n+1}-1}k\mu_{k+1} z^{k-1}
\right\|^p_{H^p}
\\ &  \asymp\sum_{n=0}^\infty 2^{n}\left\|\sum_{k=2^n}^{2^{n+1}-1}\mu_{k+1} z^{k-1} \right\|^p_{H^p}
\\ & \lesssim \sum_{n=0}^\infty 2^{n} \mu^p_{2^n}
\left\|\sum_{k=2^n}^{2^{n+1}-1}z^{k-1} \right\|^p_{H^p}.
\end{split}\end{equation*}
We claim that
\begin{equation}\label{bloque}
\left\|\sum_{k=2^n}^{2^{n+1}-1}z^{k-1} \right\|^p_{H^p}\asymp
2^{n(p-1)}.
\end{equation}
Then, using again that $\{\mu_n\}$ is decreasing
\begin{equation*}\begin{split}
||h_\mu||^p_{B^p} &\lesssim \sum_{n=0}^\infty 2^{np} \mu^p_{2^n}
 \lesssim \sum_{n=0}^\infty 2^{n(p-1)} \sum_{k=2^n-1}^{2^n} \mu^p_{k}
 \asymp \sum_{k=0}^\infty (k+1)^{p-1}\mu_k^p.
\end{split}\end{equation*}
\par An analogous reasoning using the left hand inequality in \cite[Lemma $3.4$]{LNP} proves that
\begin{equation*}\begin{split}
||h_\mu||^p_{B^p} \gtrsim \sum_{n=0}^\infty (k+1)^{p-1}\mu_k^p.
\end{split}\end{equation*}

Finally, we shall prove \eqref{bloque}. By  \cite[Lemma $3.1$]{MP}
and the M. Riesz projection theorem, it follows that
    \begin{equation}\label{eq:j11}
    \begin{split}
    \left\|\sum_{k=2^n}^{2^{n+1}-1}z^{k-1} \right\|^p_{H^p}
    \lesssim M^p_p\left(1-\frac{1}{2^n},\sum_{k=2^n}^{2^{n+1}-1}z^{k-1}\right)
    \lesssim  M^p_p\left(1-\frac{1}{2^n},\frac{1}{1-z} \right) \asymp 2^{n(p-1)}.
    \end{split}
    \end{equation}
On the other hand, using \cite[Lemma~3.1]{MP}, we obtain
    \begin{equation*}
    \begin{split}
  & M_\infty\left(1-\frac{1}{2^n}, \sum_{k=2^n}^{2^{n+1}-1}z^{k-1} \right) \asymp \left\|\sum_{k=2^n}^{2^{n+1}-1}z^{k-1} \right\|_{H^\infty}=2^n.
    \end{split}
    \end{equation*}
Furthermore, using a well-know inequality, we deduce that
    \begin{equation*}
    \begin{split}
    M_\infty\left(1-\frac{1}{2^n}, \sum_{k=2^n}^{2^{n+1}-1}z^{k-1} \right)
    &\lesssim \left(\frac{1}{2^n}-\frac{1}{2^{n+1}}\right)^{-1/p}
    M_p\left(1-\frac{1}{2^{n+1}}, \sum_{k=2^n}^{2^{n+1}-1}z^{k-1} \right)
    \\ & \lesssim \left(\frac{1}{2^n}\right)^{-1/p}
    \left\|\sum_{k=2^n}^{2^{n+1}-1}z^{k-1} \right\|_{H^p},
    \end{split}
    \end{equation*}
that is, $ \left\|\sum_{k=2^n}^{2^{n+1}-1}z^{k-1}
\right\|^p_{H^p}\gtrsim 2^{n(p-1)}$, which together with
\eqref{eq:j11} implies \eqref{bloque}. This finishes the proof.
\end{Pf}

\end{document}